\documentclass[12pt]{article}

\usepackage{amsthm}
\usepackage{amsmath, amssymb, amsfonts}
\usepackage{enumitem}
\usepackage{tikz}

\newtheorem{theorem}{Theorem}
\newtheorem*{theorem*}{Theorem}
\newtheorem{lemma}[theorem]{Lemma}
\newtheorem{corollary}[theorem]{Corollary}
\theoremstyle{definition}
\newtheorem{definition}[theorem]{Definition}

\numberwithin{theorem}{section}

\usepackage{hyperref}

\hypersetup{
  colorlinks   = true,
  urlcolor     = blue, 
  linkcolor    = blue,
  citecolor   = red
}

\usepackage{graphicx}
\usepackage{fullpage} 
\usepackage{parskip} 

\usepackage{setspace}
\newcommand{\Tmix}{\mathsf{T}_{\mathsf{mix}}}

\newcommand{\Treset}{\mathsf{T}_{\mathsf{reset}}}
\newcommand{\Tbestmix}{\mathsf{T}_{\mathsf{bestmix}}}

\newcommand{\newG}{G^*}
\newcommand{\newnewG}{G^{**}}
\newcommand{\newH}{H^*}

\newcommand{\newpi}{\pi^*}

\newcommand{\newdeg}{\deg^*}
\newcommand{\tree}[2]{\mathcal{T}_{#1,#2}}
\newcommand{\double}[2]{\mathcal{D}_{#1,#2}}

\newcommand{\pess}[1]{#1'}

\newcommand{\Tnd}{\tree{n}{d}}

\newcommand{\Dnd}{D_{n,d}}
\newcommand{\Dns}{D_{n,s}}

\DeclareMathOperator {\Ret}{Ret}

\begin{document}

\title{The Exact Mixing Time for Trees with Fixed Diameter}

\author{Andrew Beveridge, Kristin Heysse, Rhys O'Higgins and Lola Vescovo \\
Department of Mathematics, Statistics and Computer Science \\
Macalester College \\
Saint Paul, MN 55105 }

\date{}

\maketitle

\begin{abstract}
We characterize the extremal structure for the exact mixing time for random walks on trees $\tree{n}{d}$ of order $n$ with diameter $d$. Given a graph $G=(V,E)$,  let $H(v,\pi)$ denote the expected length of an optimal stopping rule from vertex $v$ to the stationary distributon $\pi$. We show that the quantity $\max_{G \in \tree{n}{d} } \Tmix(G) = \max_{G \in \tree{n}{d} }  \max_{v \in V} H(v,\pi)$ is achieved uniquely by the balanced double broom. 
\end{abstract}



%
%

\section{Introduction}

Random walks on graphs have been investigated for well over a century. These diffusion problems have important contemporary applications, such as understanding the flow of information across networks, community detection and ranking the relative importance of the network vertices; see Masuda et al.~\cite{MPL2017} for a survey of network applications.

Let $G= (V(G),E(G))$ be a connected graph.  A \emph{random walk} on $G$ starting at vertex $w$ is a sequence of vertices $w=w_0, w_1, \ldots, w_t, \ldots$ such that  for $t \geq 1$, we have  
$$
\Pr(w_{t+1} = v \mid w_t = u \} =
\left\{ 
\begin{array}{cl}
1/\deg(u) & \mbox{if } (u,v) \in E(G), \\
 0  & \mbox{otherwise}.
 \end{array}
 \right.
 $$
The  \emph{hitting time} $H(u,v)$ is the expected number of steps before a walk started at $u$ hits vertex $v$, where we define $H(u,u)=0$ for the case $u=v$. 

For non-bipartite $G$, the distribution of $w_t$ converges to the \emph{stationary distribution} $\pi$, where $\pi_v = \deg(v)/2|E|$.  For bipartite $G$, we have the same convergence if we consider a \emph{lazy walk}, in which we remain at the current state with probability $1/2$, at the cost of doubling the expected length of any walk. Herein, we will consider non-lazy walks on trees because they are are simpler to analyze;  our main result also holds for lazy walks, though we need to multiply our expressions by a factor of two.   For an introduction to random walks on graphs, see  H\"aggstr\"om \cite{Haggstrom2002} or Lov\'asz \cite{Lovasz1996}.

Recent decades have witnessed an extensive study of  extremal graph structures for random walks, and for trees in particular. 
For example, Brightwell and Winkler \cite{BW1990} showed that among trees of order $n$, the star has the minimum cover time (starting from the center) and the path has the maximum cover time (starting from a leaf).  Beveridge and Wang \cite{BW2013} 
 identified the star and the path as the extremal structures for the \emph{exact mixing time}, which is the length of an optimal stopping rule that obtains an exact sample from the stationary distribution $\pi$. In addition to fixing the order $n$ of the graph, we can also fix the diameter $d = \max_{u,v \in V} d(u,v)$, where $d(u,v)$ is the distance between vertices $u$ and $v$.

 \begin{definition}
     The family of trees of order $n$ with diameter $d$ is denoted by $\tree{n}{d}$.
 \end{definition}

 Ciardo, Dahl and Kirkland \cite{CDK2022} proved upper and lower bounds on Kemeny's constant 
 $$\sum_{u \in V} \sum_{v \in V} \pi_u \pi_v H(u,v)$$ 
 for trees in  $\tree{n}{d}$.  We contribute to this latest effort and characterize the extremal structures in $\tree{n}{d}$ for the exact mixing time of an optimal stopping rule. 
Given an initial distribution $\sigma$ and a target distribution $\tau$, a $(\sigma,\tau)$-\emph{stopping rule} $\Gamma$ halts a random walk whose initial state is drawn from $\sigma$ so that the distribution of the final state is $\tau$. or a given pair of distributions, there are a variety of $(\sigma,\tau)$-stopping rules. See \cite{LW1995} for details.

Our main result concerns a mixing measure for exact stopping rules. We  generalize the notion of a hitting time $H(u,v)$ between pairs of vertices to the \emph{access time}  $H(\sigma,\tau)$ between pairs of distributions on $V(G)$. Given a $(\sigma,\tau)$-stopping rule $\Gamma$, we let $\mathbb{E}(\Gamma)$ denote the expected number of steps of this stopping rule. The access time
$$
H(\sigma, \tau) = \min \{
\mathbb{E}(\Gamma) : \Gamma \mbox{ is a $(\sigma,\tau)$-stopping rule}
\}.
$$
is the minimum expected length of stopping rules that achieve $\tau$ when started from $\sigma$. A $(\sigma,\tau)$-stopping rule is \emph{optimal} if it achieves this minimum expected length. The most important  measure for exact stopping rules on graph $G$ is the \emph{mixing time}
\begin{equation}
\label{eqn:mixing-time-def}
\Tmix(G) = \max_{v \in V} H(v,\pi)
\end{equation}
where, for convenience, we allow  $v$ to denote the singleton distribution concentrated on that vertex. In other words, the (exact) mixing time for graph $G$ is the expected time to obtain a sample from $\pi$ via an optimal stopping rule, starting from the worst possible initial vertex. 
This mixing time leads to a family of extremal problems: given a collection of graphs, which one maximizes, or minimizes, the mixing time? 

Beveridge and Wang \cite{BW2013} prove that for trees on $n$ vertices, we have
\begin{equation}
\label{eqn:mixing-all-trees} 
\frac{3}{2} \leq \Tmix(G) \leq \frac{2n^2-4n+3}{6},
\end{equation}
where the minimum is achieved uniquely by the star $S_n=K_{1,n-1}$ and the maximum is achieved uniquely by the path $P_n$. Herein, we consider a refinement of this question by further restricting ourselves to trees of order $n$ with fixed diameter $d$. Our maximizing tree structure will be a \emph{balanced double broom} graph.

\begin{definition}
\label{def:broom}
A double broom $G \in \tree{n}{d}$  consists of a path $v_1, \ldots, v_{d-1}$ with $\ell \geq 1$ (resp. $r \geq 1$) pendant edges incident with $v_1$ (resp. $v_{d-1}$), where we label one of these leaves as $v_0$ (resp.~$v_d$). 
We use $\double{n}{d}$ to denote the set of double brooms on $n$ vertices with diameter $d$. Note that such double brooms satisfy $n=\ell+r+d+1$. Finally, we define $\Dnd \in \double{n}{d}$ to be the \emph{balanced double broom} which satisfies $\ell = \lceil (n-d-1)/2 \rceil$ and  $r = \lfloor(n-d-1)/2 \rfloor$. 
\end{definition}

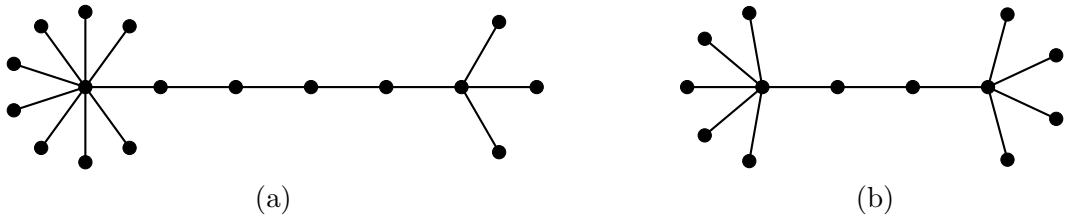
\begin{figure}[ht]
\begin{center}
\begin{tikzpicture}[scale=0.7]

\begin{scope}

\draw[thick] (0,0) -- (6,0);

\foreach \x in {0,1,2,3,4,5,6}
{
\draw[fill] (\x,0) circle (2.5pt);

}

\foreach \x in {54, 90, 126, 162, 198, 234, 270, 306}
{
\draw[thick] (0,0) -- (\x:1);
\draw[fill] (\x:1) circle (2.5pt);
}

\begin{scope}[shift={(5,0)}]

\foreach \x in {60, -60}
{
\draw[thick] (0,0) -- (\x:1);
\draw[fill] (\x:1) circle (2.5pt);
}

\end{scope}

\node at (2.5,-1.5) {\small (a)};

\end{scope}


\begin{scope}[shift={(9,0)}]

\draw[thick] (0,0) -- (3,0);

\foreach \x in {0,1,2,3}
{
\draw[fill] (\x,0) circle (2.5pt);

}

\foreach \x in {100,140,180,220,260}
{
\draw[thick] (0,0) -- (\x:1);
\draw[fill] (\x:1) circle (2.5pt);
}

\begin{scope}[shift={(3,0)}]

\foreach \x in {75, 25, -25, -75}
{
\draw[thick] (0,0) -- (\x:1);
\draw[fill] (\x:1) circle (2.5pt);
}

\end{scope}

\node at (1.5,-1.5) {\small (b)};

\end{scope}

\end{tikzpicture}
\end{center}

\caption{(a) A double broom from $\double{15}{7}$. (b) The balanced double broom $D_{13,5}$. }
\label{fig:double-broom}

\end{figure}

See Figure \ref{fig:double-broom} for an example of a double broom and a balanced double broom. We establish the balanced double broom $\Dnd \in \tree{n}{d}$ as the unique maximizer of the expected length of an optimal stopping rule to the stationary distribution $\pi$, given that we start at the worst possible vertex, as per equation \eqref{eqn:mixing-time-formula}.

\begin{theorem}
\label{thm:broom-max-mixing}
For all $d \geq 3$, we have
$$\max_{G \in \Tnd} \Tmix(G)
= \Tmix(\Dnd)
$$
and the balanced double broom $\Dnd$ is the unique maximizer. This mixing time depends on the relative parities of $n$ and $d$:
\begin{align}
\label{eqn:broom-max-mixing-odd}
\Tmix(\Dnd) &=
\frac{(d-2)n-d+5}{2} - \frac{d^3-6d^2+8d}{6(n-1)}
& \mbox{if } n-d \mbox{ is odd}, \\
\label{eqn:broom-max-mixing-even}
\Tmix(\Dnd) &=
\frac{(d-2)n-d+5}{2} - \frac{d^3-6d^2+11d-6}{6(n-1)}
&\mbox{if $n-d$ is even}.
\end{align}
\end{theorem}

The organization of this paper is as follows. In Section \ref{sec:prelim}, we recall previous results about stopping rules and random walks on trees. In Section \ref{sec:double-broom}, we derive formulas for the mixing times of double brooms. Section \ref{sec:mixing-time} contains the proof of our main theorem. We conclude with some open problems in Section \ref{sec:conclusion}

%
%

\section{Preliminaries}
\label{sec:prelim}

For an introduction to the  theory of exact stopping rules, see Lov\'asz and Winkler \cite{LW1995}. 
Herein, we explain  how to calculate the access times $H(\pi,v)$ and  $H(v,\pi)$, where $v$ is the singleton distribution on vertex $v$.    We then describe  the naive rule, which is an example of a stopping rule that   obtains exact samples from a target distribution.

Complementary to the hitting time, the \emph{return time} $\Ret(u)$ is the expected number of steps required for a random walk started at $u$ to first return to $u$. It is well-known (see Lov\'asz \cite{Lovasz1996}) that
\begin{equation}
\label{eqn:return-time}
\Ret(u) = \frac{2|E(G)|}{\deg(u)}.
\end{equation}

Turning to access times, 
when $\sigma=\pi$ and $\tau = v$ is a singleton, then the expected length of an optimal $(\pi,v)$-rule is
$$
H(\pi,v) = \sum_{u \in V} \pi_u H(u,v).
$$
Indeed, the stopping rule ``walk until you first reach $v$'' is optimal (no other rule could be faster). Meanwhile, the formula for the \emph{mixing time} 
 $H(v,\pi)$ is more subtle;  we start by defining the term \emph{pessimal}.

\begin{definition}
\label{def:pessimal}
A \emph{$v$-pessimal vertex} $\pess{v}$ is a vertex that  satisfies
$H(v',v)=\max_{w\in V}{H(w,v)}.$
\end{definition}
Note that pessimal vertices are not necessarily unique; for example, every leaf of the star $S_n$ is pessimal for the central vertex. We can now give the formula for $H(v,\pi)$  for any graph $G$.
\begin{theorem}[Theorem 5.2 and equation (8) in \cite{LW1995}]
\label{thm:mixing-time}
The expected length of an optimal mixing rule starting from vertex $v$ is 
\begin{equation}
\label{eqn:mixing-time-formula}
H(v, \pi) 
= H(v',v) - H(\pi,v)
= H(\pess{v}, v) - \sum_{u \in V} \pi_u H(u,v).
\end{equation}
\end{theorem}

\begin{theorem}[Corollary 4.2 in \cite{LW1998} applied to an undirected graph] 
\label{thm:mixing-duality}
If $z \in V$ achieves 
$$
\Tmix(G) = H(z,\pi) = H(z',z) - H(\pi, z)
$$ 
on an undirected graph, then so does the $z$-pessimal vertex $z'$. Moreover, $z$ is a $z'$-pessimal vertex, that is $H(z,z') = \max_{v \in V} H(v,z')$, and
$$
\Tmix(G) = H(z',\pi) = H(z,z') - H(\pi, z').
$$
\end{theorem}

If we know all of the vertex-to-vertex hitting times, then we can use  Theorem \ref{thm:mixing-time} to calculate the mixing time $H(v, \pi)$ for each  starting vertex $v \in V$. Theorem \ref{thm:mixing-duality} will be useful for technical reasons, since it guarantees that there is aways a pessimal pair of vertices, both of which achieve $\Tmix(G)$ for an undirected graph.

In general, we can determine every hitting time $H(u,v)$ by solving a system of linear equations. However, the acyclic structure of trees leads to  an explicit formula in terms of the distances and degrees of the graph. We follow the notation of  \cite{beveridge2009,BW2013}, but equivalent variations on this tree formula abound in the literature (c.f.~\cite{DTW2003,GW2013},  and Chapter 5 of \cite{aldous+fill}).  From here forward, we assume that $G=(V(G),E(G))$ is a tree on $n \geq 4$ vertices.

Suppose that $(u,v) \in E(G)$ is an edge in the tree.   Removing this edge breaks $G$ into two disjoint trees $G_1$ and $G_2$, where $u \in V(G_1)$ and $v \in V(G_2)$. We define $V_{u:v} = V(G_1)$ and $V_{v:u} = V(G_2)$. We think of $V_{u:v}$ as the set of vertices that are closer to $u$ than to $v$. For these adjacent vertices $u$ and $v$, let $F$ denote the induced tree on $V_{u:v} \cup \{   v \}$. We have
\begin{equation} 
\label{eqn:adjhtime}
		H(u,v)= \Ret_{F} (v) - 1 = \sum_{w\in V_{u:v}}{\deg(w)} = 2|V_{u:v}|-1,
\end{equation}
where the first equality holds by equation \eqref{eqn:return-time}.
Equation \eqref{eqn:adjhtime} encodes a very useful property of trees: the hitting time from a vertex $u$ to an adjacent vertex $v$ only depends on $|V_{u:v}|$, independent of the particular structure of the tree.
This equation can be used to  reveal that when  $G=P_n$ is the path ($v_0,v_1,\dots,v_{n-1}$), the hitting times are
\begin{equation} 
\label{eq:htimepath2}
H(v_i,v_j) =
j^2-i^2 \qquad \mbox{if } i \leq j.
\end{equation}

This formula is calculated by repeatedly using equation \eqref{eqn:adjhtime} to evaluate
$$H(v_i, v_j) = H(v_i, v_{i+1}) + H(v_{i+1}, v_{i+2}) + \cdots + H(v_{j-1}, v_{j}).$$ 
A similar argument produces a formula for the hitting times of an arbitrary tree, but first we need some additional notation. Let  $u,v,w \in V$. Recall that $d(u,v)$ is the distance between these two vertices, and and define
$$
\ell(u,v;w)=\frac{1}{2}(d(u,w)+d(v,w)-d(u,v))
$$
to be the length of the intersection of the $(u,w)$-path and the $(v,w)$-path. 

\begin{lemma}[\cite{beveridge2009}]
\label{lemma:hitting-time}
For any pair of vertices $u,v$, we have
\begin{equation} 
\label{eqn:hitting-time}
H(u,v)=\sum_{w\in V}{\ell(u,w;v)\deg(w)}.
\end{equation}
\end{lemma}
Theorem \ref{thm:mixing-time} and Lemma \ref{lemma:hitting-time} will be foundational for our proof of Theorem \ref{thm:broom-max-mixing}. Together, they provide a simple way to calculate hitting times and mixing times for trees.

\subsection{The naive rule}

What does a stopping rule look like in practice?
Intriguingly, we typically have \emph{multiple} optimal mixing rules from $v$ to $\pi$.
  Lov\'asz and Winkler  \cite{LW1995a} describe four distinct optimal rules,  each of which achieves the minimum expected length of an $(v,\pi)$-rule. We choose to describe the \emph{naive rule} because it is the simplest stopping rule. (However, the naive rule it is rarely optimal: usually, there are other rules with smaller expected lengths.) 
To avoid going deeper into the theory of stopping rules, we simply mention one useful characterization of  optimal rules.

\begin{theorem}[Theorem 5.1 in \cite{LW1995}]
A stopping rule is optimal if and only if there is at least one vertex that is a \emph{halting state}, meaning that  if the random walk reaches this vertex, then the optimal stopping rule \emph{always} stops the walk.    
\end{theorem}

We start by describing the naive rule from $v_0$ to $\pi$ on the  path $P_4$ on vertices $(v_0, v_1, v_2, v_3)$. In the naive rule, we choose our target vertex up front (with the appropriate probability), and walk to that pre-chosen vertex. Our target probabilities are $(1/6,1/3,1/3,1/6)$ so the expected length of this naive stopping rule is
$$
\frac{1}{6} H(v_0,v_0) + \frac{1}{3} H(v_0,v_1) + \frac{1}{3} H(v_0,v_2) + \frac{1}{6} H(v_0,v_3) = 0 + \frac{1}{3} + \frac{4}{3} = \frac{9}{6} = \frac{19}{6}.
$$
This naive rule is optimal because $v_3$ is a halting state. Indeed, we get the same $19/6$ value  from formula \eqref{eqn:mixing-time-formula} via a different calculation.

Next, let's consider a naive rule that is not optimal. This time, we follow the naive rule from leaf $w_1$ to $\pi$ on the star graph $S_4 = K_{1,3}$ on vertices $\{ w_0,w_1, w_2, w_3\}$, where $w_0$ is the central vertex. Our target probabilities are $(1/2, 1/6, 1/6, 1/6)$ and the expected length of the rule is
$$
\frac{1}{2} H(w_1,w_0) + \frac{1}{6} H(w_1,w_1) + \frac{1}{6} H(w_1,w_2) + \frac{1}{6} H(w_1,w_3)
= \frac{1}{2} + 0 + \frac{6}{6} +   \frac{6}{6} = \frac{5}{2}.
$$
This naive rule is not optimal because it does not have a halting state. For comparison, here is an optimal $(w_1,\pi)$-stopping rule whose expected length is $3/2$: take one step, then take a second step with probability $1/2$.

%
%

\section{The Mixing Time for a Double Broom}
\label{sec:double-broom}

Let $G \in \double{n}{d}$ be a double broom with geodesic path $v_0, v_1, \ldots, v_d$.
We derive formulas for the hitting times $H(u,v)$, the access time $H(\pi, v_d)$ and the mixing time $\Tmix(G) = \max_v H(v,\pi)$. Having the mixing time formulas at our disposal will be  essential for the inductive proof of our main result in Section \ref{sec:end}.

\begin{lemma}
\label{lemma:double-broom-hitting-times}
Let $G \in \double{n}{d}$ be a double broom with $\ell$ left leaves, $r$ right leaves, and  geodesic $v_0, v_1, \ldots, v_d$, so that $\ell+r=n-d+1$.
We have
\begin{align*}
H(v_0, v_k) &= k^2 + (\ell-1)(k-1), \quad \mbox{for }  1 \leq k < d, \\
H(v_0, v_d)  &= d^2 + 2(\ell-1) (d-1) + 2(r-1), \\
H(v_k,v_d)  &= d^2  - k^2 + 2(\ell -1)(d-k) + 2(r-1), \quad \mbox{for }   1 \leq k \leq d-1.
\end{align*}
\end{lemma}

\begin{proof}
The first two equations are calculated using equation \eqref{eqn:hitting-time}. The third equation follows directly because $H(v_k,v_d)  = H(v_0, v_d) - H(v_0,v_k).$ 
\end{proof}

\begin{lemma}
\label{lemma:double-broom-pi-access-times}
Let $G \in \double{n}{d}$ be a double broom with $\ell$ left leaves and $r$ right leaves. Then
\begin{align*}
H(\pi,v_d) &= \frac{1}{6 (n-1)} \bigg(
4 d^3+12 d^2 (\ell-1)+d (12 \ell(\ell-2) +24 r-13) \\
& \qquad
+3 \left(-4 \ell^2+\ell (8 r-3)+r (4
   r-19)+14\right) \bigg).
\end{align*}
\end{lemma}

\begin{proof}
First, we note that $n=(d+1)+(\ell-1)+(r-1)=d+\ell+r-1$. We have
\begin{align*}
H(\pi,v_d) &= \sum_{v \in V} \pi_v H(v, v_d)\\
&= \frac{1}{2(n-1)}  \bigg( \ell \, H(v_0, v_d)
+ (\ell+1) H(v_1, v_d)
+ \sum_{k=2}^{d-2} 2 H(v_k, v_d) \\
&\qquad 
+ (r+1) H(v_{d-1}, v_d)
+ (r-1) (1 + H(v_{d-1}, v_d))
\bigg).
\end{align*}
The proof follows from the hitting time formulas in Lemma \ref{lemma:double-broom-hitting-times}. The calculations can be checked using a mathematical software system, such as  Sage or Mathematica.
\end{proof}

\begin{lemma}
\label{lemma:double-broom-mixing-time}
Let $G \in \double{n}{d}$ be a double broom with $\ell$ left leaves and $r$ right leaves. Then
\begin{align*}
\quad \Tmix(G) 
&= \frac{1}{6 (n-1)} 
\bigg( 2 d^3+6 d^2 (\ell+r-2)+d (12 \ell (r-2)-24 r+37) \\
& \qquad +\ell (33-24 r)+33 r-42 \bigg). 
\end{align*}
\end{lemma}

\begin{proof} Recall that out double broom $G$ consists of a path $v_1, v_2, \ldots, v_{d-1}$, with $\ell$ leaves added to $v_1$ and $r$ leaves added to $v_{d-1}$. For convenience, we label one of the $\ell$ left leaves as $v_0$, and one of the $r$ right leaves as $v_{d}$. 

By Theorem \ref{thm:mixing-duality}, there are two vertices $x$ and $y$ of $G$, such that (a) $\Tmix=H(x,\pi) = H(y,\pi)$, and furthermore that (b) $x'=y$ and $y'=x$. It is clear that $x$ and $y$ must both be leaves, since otherwise (b) does not hold. We quickly realize that $x=v_0$ and $y=v_d$ (or another pair of left and right leaves) is the only sensible choice.

We have $\Tmix(G) = H(v_0,\pi) = H(v_0, v_d) - H(\pi, d)$ by equation \eqref{eqn:mixing-time-formula}. The stated formula follows from Lemmas 
\ref{lemma:double-broom-hitting-times} and \ref{lemma:double-broom-pi-access-times}.  Once again, the calculations are best checked using a mathematical software system.
\end{proof}

Let us apply this lemma to paths and to double stars. The path $P_n$ is the balanced double broom $D_{n,n-1}$ with diameter $d=n-1$ and with $\ell=r=1$.  The mixing time formula of Lemma \ref{lemma:double-broom-mixing-time}  simplifies to 
$$
\Tmix(P_n) = \Tmix(D_{n,n-1})  = \frac{2 n^2 -4n + 3 }{6},
$$
which matches the path mixing time originally calculated in \cite{BW2013}. At the other extreme, each tree $G \in \tree{n}{3} $ is a double star, which is a double broom with diameter $d=3$. A double star with  $\ell$ left leaves and $r$ right leaves has   $n=2+\ell+r$ vertices in total. For such a double star, we have
\begin{equation}
\label{eqn:double-broom-diameter-3}
\Tmix(G) = \frac{4 \ell r+5 \ell+5 r+5}{2 (\ell+r+1)}
= \frac{4 \ell r+5(n-1)}{2 (n-1)} =  \frac{4 \ell r}{2 (n-1)} + \frac{5}{2}.
\end{equation}

We conclude this section with the mixing time formula for balanced double brooms, which appear in Theorem \ref{thm:broom-max-mixing}. The  formula depends on  the parity of  $n-d$, which determines whether there are an even or an odd number of leaves. 

\begin{corollary}
\label{cor:balanced-double-broom-mixing-time}
The mixing time of the balanced double broom $D_{n,d}$ is
\begin{align}
\label{eqn:balanced-mixing-case1}
\Tmix(D_{n,d})  &= 
\frac{(d-2)n-d+5}{2} - \frac{d^3-6d^2+8d}{6(n-1)}
& \mbox{if $n-d$ is odd},
\\
\label{eqn:balanced-mixing-case2}
\Tmix(D_{n,d}) &=
\frac{(d-2)n-d+5}{2} - \frac{d^3-6d^2+11d-6}{6(n-1)}
 & \mbox{if $n-d$ is even}.
\end{align}
\end{corollary}

\begin{proof}
When $n-d$ is odd, we have $\ell = (n-d+1)/2 = r$. When $n-d$ is even, we have $\ell = (n-d+2)/2$ and  $r=(n-d)/2$. The two formulas follow via substitution into the equation in Lemma \ref{lemma:double-broom-mixing-time}.  
\end{proof}

For example, when $d = 3$, the double broom $D_{n,3}$ is a balanced double star. Equations \eqref{eqn:balanced-mixing-case1} and \eqref{eqn:balanced-mixing-case2}, simplify to 
$$
\Tmix(D_{n,3}) = 
\left\{
\begin{array}{cl}
\frac{n}{2}+1 + \frac{1}{2(n-1)}& \mbox{for $n$ even} ,\\
\frac{n}{2}+1 & \mbox{for $n$ odd}. \\
\end{array}
\right.
$$

The next corollary about double brooms will play a crucial role in our proof of our main result.

\begin{corollary}
\label{cor:balanced-double-brooms-increasing}
For a fixed order $n$ and diameters $3 \leq d_1 < d_2 \leq n-1$, we have $\Tmix(D_{n,d_1}) < \Tmix(D_{n,d_2})$.  
\end{corollary}

\begin{proof}
If we increase the diameter from $d$ to $d+1$ and go from an even number of leaves to an odd number of leaves, then the change in the mixing time is equation \eqref{eqn:balanced-mixing-case1} for $D_{n,d+1}$ minus equation \eqref{eqn:balanced-mixing-case2}, which simplifies to
$$
\Tmix(D_{n,d+1}) - \Tmix(D_{n,d})  = \frac{-d^2+2 d+(n-1)^2}{2 (n-1)} = \frac{n-1}{2} - \frac{d^2-2d}{2(n-1)} > 0.
$$
If we increase the diameter from $d$ to $d+1$ and go from an odd number of leaves to an even number of leaves, then the change in the mixing time is equation \eqref{eqn:balanced-mixing-case2} for $D_{n,d+1}$ minus equation \eqref{eqn:balanced-mixing-case1}, which simplifies to
\begin{align*}
\Tmix(D_{n,d+1}) - \Tmix(D_{n,d})  &=  \frac{-d^2+4 d+n^2-2 n-2}{2 n-2}  \\
&= \frac{n-1}{2} - \frac{d^2-2d+3}{2(n-1)} > 0.
\end{align*}
In other words, fixing $n$, increasing the diameter also increases the mixing time of the corresponding balanced double broom.
\end{proof}

%
%

\section{Maximizing the Mixing Time}
\label{sec:mixing-time}

We embark on proving that $\max_{G \in \tree{n}{d}} \Tmix(G) = \Tmix(\Dnd)$ where $\Dnd$ is the balanced double broom. 
Recall that a \emph{caterpillar} is a tree $G$ in which every vertex is within distance one of a central path $P=\{v_0, v_1, \ldots, v_d\}$, called the \emph{spine} of $G$.

Starting from any tree $G \neq \Dnd$, we make incremental changes to the tree that increase the mixing time. This evolutionary process completes when we have created the balanced double broom $\Dnd$.
Our evolution consists of three phases, as illustrated in Figure \ref{fig:example}.
In Phase One, we convert our tree into a caterpillar, and perhaps reduce the diameter to $s < d$. 
In Phase Two, we move leaves outwards in pairs, which results in either a double broom, or a double broom with one additional leaf located somewher on the spine. In Phase Three, we move this additional leaf (if it exists), and balance the two sides of the double broom. Finally, if we reduced the diameter in phase one, we simply observe that $\Tmix(\Dns) < \Tmix(\Dnd)$. 

In the subsections that follow, we show that the mixing time increases at every step of the evolution. We then present the proof of Theorem \ref{thm:broom-max-mixing}.
First, we gather some necessary definitions and operations.

\begin{figure}

\begin{center}

\begin{tikzpicture}[scale=0.8]

\def\spacer{-3}

\draw (13.66,-0.75) rectangle (14.33,-3.25) node[pos=.5, rotate=90] {Phase One};

\draw (13.66,-4.25) rectangle (14.33,-8.75) node[pos=.5, rotate=90] {Phase Two};

\draw (13.66,-10.25) rectangle (14.33,-15) node[pos=.5, rotate=90] {Phase Three};

\begin{scope}[shift={(1,0)}]

\node at (9.5,0.5) {\begin{tabular}{p{2in}} Start with a tree in $\tree{13}{6}$. \end{tabular}};

\foreach \x in {135,165,195,225}
{
\draw[fill] (\x:1) circle (2.5pt);
\draw[thick] (0,0) -- (\x:1);
}

\draw[fill] (0,0) circle (2.5pt);
\draw[fill] (1,0) circle (2.5pt);
\draw[thick] (0,0) -- (1,0);

\node[above] at (110:.8) {$v_0=z$};
\node[above] at (.2,0) {$v_1$};
\node[above] at (1,0) {$v_2$};

\begin{scope}[shift={(1,0)}]

\foreach \x in {15}
{
\draw[thick] (0,0) -- (\x:3);
\foreach \y in {1,2,3}
{
\draw[fill] (\x:\y) circle (2.5pt);
}
}

\foreach \x in {-15}
{
\draw[thick] (0,0) -- (\x:3);
\foreach \y in {1,2,3}
{
\draw[thick, fill=white] (\x:\y) circle (2.5pt);
}
}

\node[above] at (20:.9) {$v_3$};
\node[below] at (15:2) {$v_4$};

\node[below] at (14:3.5) {$v_5 =z'$};

\begin{scope}[shift={(15:1)}]
\draw[thick] (0,0) -- (45:1);
\draw[thick, fill=gray] (45:1) circle (2.5pt);

\end{scope}

\end{scope}
\end{scope}

\begin{scope}[shift={(0,\spacer)}]

\draw[-latex, very thick, gray] (6.5,1.75) to [bend left] (6.5,.25);
\node at (10.1,1) {\begin{tabular}{p{1.75in}} Create a caterpillar in $\tree{13}{5}$ with spine $v_0, v_1, \ldots, v_5$. \end{tabular}};

\foreach \x in {0,1,2,3,4,5}
{
\draw[fill] (\x,0) circle (2.5pt);
\node[above] at (\x,0) {$v_{\x}$};
}

\draw[thick] (0,0) -- (5,0);

\begin{scope}[shift={(1,0)}]

\foreach \x in {-75, -90, -105}
{
\draw[thick] (0,0) -- (\x:1);
\draw[fill] (\x:1) circle (2.5pt);

}

\end{scope}

\begin{scope}[shift={(2,0)}]

\foreach \x in {-75, -90, -105}
{
\draw[thick] (0,0) -- (\x:1);
\draw[thick, fill=white] (\x:1) circle (2.5pt);

}

\end{scope}

\draw[thick] (3,0) -- (3,-1);
\draw[thick, fill=gray] (3,-1) circle (2.5pt);

\end{scope}

\begin{scope}[shift={(0,2*\spacer)}]

\draw[-latex, very thick, gray] (6.5,1.75) to [bend left] (6.5,.25);
\node at (10.1,1) {\begin{tabular}{p{1.75in}} Move a $v_2$-leaf to $v_1$ and move the $v_3$-leaf to $v_4$. \end{tabular}};

\foreach \x in {0,1,2,3,4,5}
{
\draw[fill] (\x,0) circle (2.5pt);
\node[above] at (\x,0) {$v_{\x}$};
}

\draw[thick] (0,0) -- (5,0);

\begin{scope}[shift={(1,0)}]

\foreach \x in {-82,-98,-114}
{
\draw[thick] (0,0) -- (\x:1);
\draw[fill] (\x:1) circle (2.5pt);
}

\foreach \x in {-66}
{
\draw[thick] (0,0) -- (\x:1);
\draw[thick, fill=white] (\x:1) circle (2.5pt);
}

\end{scope}

\begin{scope}[shift={(2,0)}]

\foreach \x in {-82,-98}
{
\draw[thick] (0,0) -- (\x:1);
\draw[thick, fill=white] (\x:1) circle (2.5pt);
}

\end{scope}

\draw[thick] (4,0) -- (4,-1);
\draw[thick, fill=gray] (4,-1) circle (2.5pt);

\end{scope}

\begin{scope}[shift={(0,3*\spacer)}]

\draw[-latex, very thick, gray] (6.5,1.75) to [bend left] (6.5,.25);
\node at (10.1,1) {\begin{tabular}{p{1.75in}} Move a $v_2$-leaf to $v_1$ and move the other $v_2$-leaf to $v_3$ to create a near double broom. \end{tabular}};

\foreach \x in {0,1,2,3,4,5}
{
\draw[fill] (\x,0) circle (2.5pt);
\node[above] at (\x,0) {$v_{\x}$};
}

\draw[thick] (0,0) -- (5,0);

\begin{scope}[shift={(1,0)}]

\foreach \x in {-90,-106,-122}
{
\draw[thick] (0,0) -- (\x:1);
\draw[fill] (\x:1) circle (2.5pt);
}

\foreach \x in {-58,-74}
{
\draw[thick] (0,0) -- (\x:1);
\draw[fill=white] (\x:1) circle (2.5pt);
}

\end{scope}

\draw[thick] (3,0) -- (3,-1);
\draw[thick, fill=white] (3,-1) circle (2.5pt);

\draw[thick] (4,0) -- (4,-1);
\draw[thick, fill=gray] (4,-1) circle (2.5pt);

\end{scope}

\begin{scope}[shift={(0,4*\spacer)}]

\draw[-latex, very thick, gray] (6.5,1.75) to [bend left] (6.5,.25);
\node at (10.1,1) {\begin{tabular}{p{1.75in}} Move the $v_3$ leaf to $v_4$ to create a double broom. \end{tabular}};

\foreach \x in {0,1,2,3,4,5}
{
\draw[fill] (\x,0) circle (2.5pt);
\node[above] at (\x,0) {$v_{\x}$};
}

\draw[thick] (0,0) -- (5,0);

\begin{scope}[shift={(1,0)}]

\foreach \x in {-90,-106,-122}
{
\draw[thick] (0,0) -- (\x:1);
\draw[fill] (\x:1) circle (2.5pt);
}

\foreach \x in {-58,-74}
{
\draw[thick] (0,0) -- (\x:1);
\draw[fill=white] (\x:1) circle (2.5pt);
}

\end{scope}

\begin{scope}[shift={(4,0)}]

\foreach \x in {-82}
{
\draw[thick] (0,0) -- (\x:1);
\draw[thick, fill=gray] (\x:1) circle (2.5pt);
}

\foreach \x in {-98}
{
\draw[thick] (0,0) -- (\x:1);
\draw[fill=white] (\x:1) circle (2.5pt);
}

\end{scope}

\end{scope}

\begin{scope}[shift={(0,5*\spacer)}]

\draw[-latex, very thick, gray] (6.5,1.75) to [bend left] (6.5,.25);
\node at (10.1,1) {\begin{tabular}{p{1.75in}} Move a $v_1$ leaf to $v_4$ to create the balanced double broom $D_{13,5}$.\end{tabular}};

\foreach \x in {0,1,2,3,4,5}
{
\draw[fill] (\x,0) circle (2.5pt);
\node[above] at (\x,0) {$v_{\x}$};
}

\draw[thick] (0,0) -- (5,0);

\begin{scope}[shift={(1,0)}]

\foreach \x in {-64}
{
\draw[thick] (0,0) -- (\x:1);
\draw[fill=white] (\x:1) circle (2.5pt);
}

\foreach \x in {-82,-98,-114}
{
\draw[thick] (0,0) -- (\x:1);
\draw[fill] (\x:1) circle (2.5pt);
}

\end{scope}

\begin{scope}[shift={(4,0)}]

\foreach \x in {-74}
{
\draw[thick] (0,0) -- (\x:1);
\draw[thick, fill=gray] (\x:1) circle (2.5pt);
}

\foreach \x in { -90, -106}
{
\draw[thick] (0,0) -- (\x:1);
\draw[thick, fill=white] (\x:1) circle (2.5pt);
}

\end{scope}

\end{scope}

\begin{scope}[shift={(0,6*\spacer)}]

\draw[-latex, very thick, gray] (6.5,1.75) to [bend left] (6.5,.25);
\node at (10.1,1) {\begin{tabular}{p{1.75in}} Replace $D_{13,5} \in \tree{13}{5}$ with $D_{13,6} \in \tree{13}{6}$ so that the diameter is the same as the original tree.\end{tabular}};

\foreach \x in {0,1,2,3,4,5,6}
{
\draw[fill] (\x,0) circle (2.5pt);
\node[above] at (\x,0) {$v_{\x}$};
}

\draw[thick] (0,0) -- (6,0);

\begin{scope}[shift={(1,0)}]

\foreach \x in {-74, -90, -106}
{
\draw[fill] (\x:1) circle (2.5pt);
\draw[thick] (0,0) -- (\x:1);
}

\end{scope}

\begin{scope}[shift={(5,0)}]

\foreach \x in {-74, -90, -106}
{
\draw[fill] (\x:1) circle (2.5pt);
\draw[thick] (0,0) -- (\x:1);
}

\end{scope}

\end{scope}

\end{tikzpicture}

\end{center}

\caption{The evolution of a tree into a balanced double broom. At each step, the mixing time increases. }
\label{fig:example}

\end{figure}

\begin{definition}
\label{def:near-double-broom}
A \emph{near double broom} is a tree that is not a double broom, but such that there exists exactly one leaf whose removal results in a double broom.
\end{definition}

\begin{definition}
Let $G$ be a tree. A \emph{pessimal path} for $G$ is a leaf-to-leaf subpath $P = \{ v_0, v_1, \ldots, v_s \}$,   such that  
$$
\Tmix = H(v_0,\pi) = H(v_s, \pi)
$$
where the existence of these two leaves that achieve $\Tmix(G)$ is guaranteed by Theorem \ref{thm:mixing-duality}.
\end{definition}

We evolve our tree by moving one or two leaves at a time. We refer to these leaf operations as \emph{tree surgeries.}

\begin{definition}
Let $G \in \tree{n}{d}$ be a non-caterpillar, with pessimal path $P = \{ v_0, v_1, \ldots, v_s \}$. Let $y \in V(G \backslash P)$ be a leaf that is not adjacent to $P$. Define the tree surgery $\sigma(G,P;y)$ to be the tree obtained by moving  leaf $y$ to be adjacent to  the pessimal path vertex $v_k$ that is closest  to $y$. 
\end{definition}

\begin{definition}
Let $G \in \tree{n}{d}$ be a caterpillar with spine $P=\{v_0, v_1, \ldots, v_d\}$. If $G$ has a leaf $x \in V(G \backslash P)$ adjacent to $v_i$, then we  define $\tau(G,P;(i,j))$ to be the caterpillar $G - (v_i, x) + (v_j, x)$. If $G$ also has a leaf $y \in V(G \backslash P)$ adjacent to $v_k$, then we define the \emph{tree surgery}
$\tau(G,P; (i,j) \wedge (k,\ell))$ to be the caterpillar $G - (v_i, x) + (v_j, x) - (v_k,y) + (v_{\ell}, y).$
\end{definition}

Examples of $\sigma(G,P;y)$, $\tau(G,P; (i,i-1) \wedge (j,j+1))$, and $\tau(G,P; (i,d-1))$ are shown in Figure \ref{fig:cat-to-double-broom}. These are the three types of tree surgeries used in our evolution from tree to balanced double broom.

\begin{figure}[ht]
\begin{center}
\begin{tikzpicture}[scale=0.7]

\begin{scope}[shift={(0,0.5)}]

\node at (-1,-0.5) {\small (a)};
\begin{scope}

\draw[thick] (0,0) -- (6,0);

\foreach \x in {0,1,2,3,4,5,6}
{
\draw[fill] (\x,0) circle (2.5pt);
\node[above] at (\x,0) {$v_{\x}$};
}

\begin{scope}[shift={(3,0)}]

\draw[thick] (0,0) -- (-70:2);
\draw[fill] (-70:1) circle (2.5pt);
\draw[fill] (-70:2) circle (2.5pt);

\begin{scope}[shift={(-70:1)}]

\draw[thick] (0,0) -- (-110:1);
\draw[fill=white] (-110:1) circle (2.5pt);
\node[left] at (-110:1) {$y$};

\draw[-latex] (-120:.9) to [bend left] (-.4,.8);

\end{scope}

\end{scope}

\draw[thick] (1,0) -- (1,-1);
\draw[fill] (1,-1) circle (2.5pt);

\begin{scope}[shift={(4,0)}]
\draw[thick] (0,0) -- (-60:2);
\draw[fill] (-60:1) circle (2.5pt);
\draw[fill] (-60:2) circle (2.5pt);

\end{scope}

\end{scope}

\draw[very thick, -latex] (6.5, -.5) -- (7.5,-.5);

\begin{scope}[shift={(8,0)}]

\draw[thick] (0,0) -- (6,0);

\foreach \x in {0,1,2,3,4,5,6}
{
\draw[fill] (\x,0) circle (2.5pt);
\node[above] at (\x,0) {$v_{\x}$};
}

\begin{scope}[shift={(3,0)}]

\draw[thick] (0,0) -- (-70:2);
\draw[thick] (0,0) -- (-110:1);

\draw[fill] (-70:1) circle (2.5pt);
\draw[fill] (-70:2) circle (2.5pt);
\draw[fill=white] (-110:1) circle (2.5pt);

\node[left] at (-110:1) {$y$};

\end{scope}

\draw[thick] (1,0) -- (1,-1);
\draw[fill] (1,-1) circle (2.5pt);

\begin{scope}[shift={(4,0)}]
\draw[thick] (0,0) -- (-60:2);
\draw[fill] (-60:1) circle (2.5pt);
\draw[fill] (-60:2) circle (2.5pt);

\end{scope}

\end{scope}
\end{scope}

\begin{scope}[shift={(0,-6)}]

\node at (-1,-0.5) {\small (c)};
\begin{scope}

\draw[thick] (0,0) -- (6,0);

\foreach \x in {0,1,2,3,4,5,6}
{
\draw[fill] (\x,0) circle (2.5pt);
\node[above] at (\x,0) {$v_{\x}$};
}

\begin{scope}[shift={(1,0)}]

\foreach \x in {-64,-82, -98, -116}
{
\draw[thick] (0,0) -- (\x:1);
\draw[fill] (\x:1) circle (2.5pt);
}

\end{scope}

\draw[thick] (2,0) -- (2,-1);
\draw[fill=white] (2,-1) circle (2.5pt);

\draw[-latex] (2.1,-1.25) to [bend right] (4.7,-1.25);

\begin{scope}[shift={(5,0)}]

\foreach \x in {-82,-98}
{
\draw[thick] (0,0) -- (\x:1);
}

\draw[fill] (-98:1) circle (2.5pt);
\draw[fill] (-82:1) circle (2.5pt);

\end{scope}

\end{scope}

\draw[very thick, -latex] (6.5, -.5) -- (7.5,-.5);

\begin{scope}[shift={(8,0)}]

\draw[thick] (0,0) -- (6,0);

\foreach \x in {0,1,2,3,4,5,6}
{
\draw[fill] (\x,0) circle (2.5pt);
\node[above] at (\x,0) {$v_{\x}$};
}

\begin{scope}[shift={(1,0)}]

\foreach \x in {-64,-82, -98, -116}
{
\draw[thick] (0,0) -- (\x:1);
\draw[fill] (\x:1) circle (2.5pt);
}

\end{scope}

\begin{scope}[shift={(5,0)}]

\foreach \x in {-72,-90,-108}
{
\draw[thick] (0,0) -- (\x:1);
}

\draw[fill] (-72:1) circle (2.5pt);
\draw[fill] (-90:1) circle (2.5pt);
\draw[fill=white] (-108:1) circle (2.5pt);

\end{scope}

\end{scope}
\end{scope}

\begin{scope}[shift={(0,-3)}]

\node at (-1,-0.5) {\small (b)};
\begin{scope}

\draw[thick] (0,0) -- (6,0);

\foreach \x in {0,1,2,3,4,5,6}
{
\draw[fill] (\x,0) circle (2.5pt);
\node[above] at (\x,0) {$v_{\x}$};
}

\begin{scope}[shift={(1,0)}]

\foreach \x in {-72, -90, -108}
{
\draw[thick] (0,0) -- (\x:1);
}

\draw[fill] (-108:1) circle (2.5pt);
\draw[fill] (-90:1) circle (2.5pt);
\draw[fill] (-72:1) circle (2.5pt);

\end{scope}

\draw[thick] (3,0) -- (3,-1);
\draw[fill=white] (3,-1) circle (2.5pt);

\draw[-latex] (2.9,-1.25) to [bend left] (2.1,-1.25);

\begin{scope}[shift={(4,0)}]

\foreach \x in {-82,-98}
{
\draw[thick] (0,0) -- (\x:1);
}

\draw[fill] (-98:1) circle (2.5pt);
\draw[fill=white] (-82:1) circle (2.5pt);

\end{scope}

\draw[-latex] (4.3,-1.25) to [bend right] (4.8,-1.25);

\draw[thick] (5,0) -- (5,-1);
\draw[fill] (5,-1) circle (2.5pt);

\end{scope}

\draw[very thick, -latex] (6.5, -.5) -- (7.5,-.5);

\begin{scope}[shift={(8,0)}]

\draw[thick] (0,0) -- (6,0);

\foreach \x in {0,1,2,3,4,5,6}
{
\draw[fill] (\x,0) circle (2.5pt);
\node[above] at (\x,0) {$v_{\x}$};
}

\begin{scope}[shift={(1,0)}]

\foreach \x in {-72, -90, -108}
{
\draw[thick] (0,0) -- (\x:1);
}

\draw[fill] (-108:1) circle (2.5pt);
\draw[fill] (-90:1) circle (2.5pt);
\draw[fill] (-72:1) circle (2.5pt);

\end{scope}

\draw[thick] (2,0) -- (2,-1);
\draw[fill=white] (2,-1) circle (2.5pt);

\draw[thick] (4,0) -- (4,-1);
\draw[fill] (4,-1) circle (2.5pt);

\begin{scope}[shift={(5,0)}]

\foreach \x in {-82,-98}
{
\draw[thick] (0,0) -- (\x:1);
}

\draw[fill=white] (-98:1) circle (2.5pt);
\draw[fill] (-82:1) circle (2.5pt);

\end{scope}

\end{scope}

\end{scope}

\end{tikzpicture}
\end{center}

\caption{(a) An example of $\sigma(G,P;y)$. (b) An example of $\tau(G,P; (3,2) \wedge (4,5))$. (c) An example of $\tau(G,P; (2,5))$.}
\label{fig:cat-to-double-broom}
\end{figure}
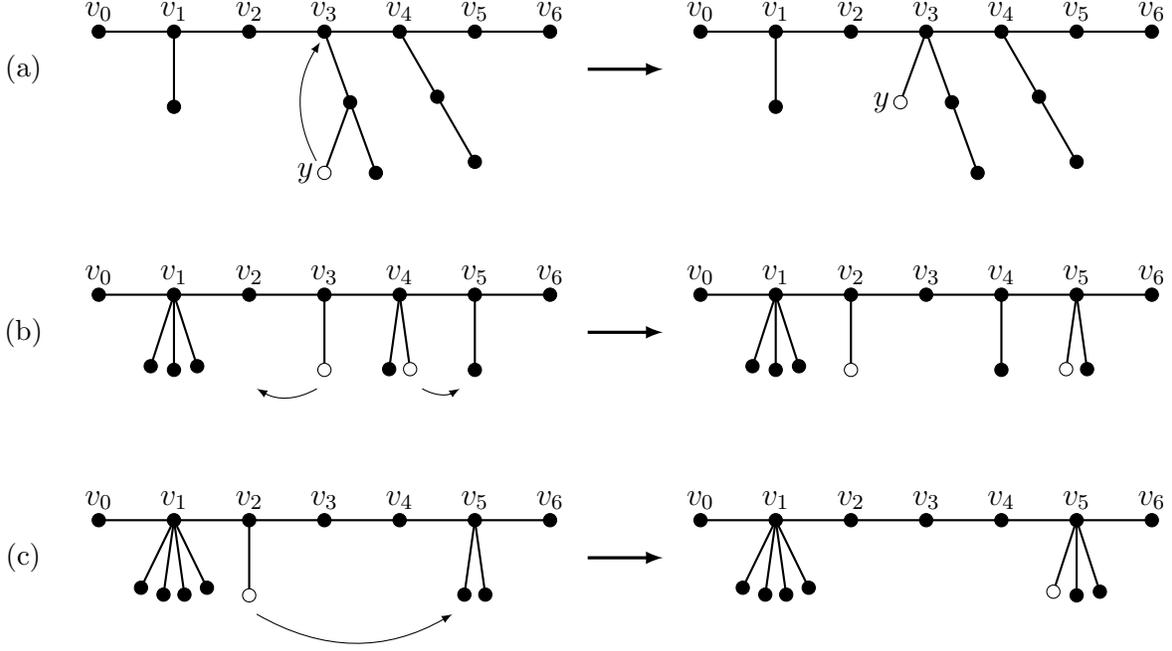

\subsection{Phase One: Tree to Caterpillar}

In this section, we show that for any tree $G \in \tree{n}{d}$ that is not a caterpillar, there exists a caterpillar $\newG \in \tree{n}{s}$, where $s \leq d$, with a larger mixing time $\Tmix(\newG) > \Tmix(G)$. From here forward, we will use starred notation to refer to quantities related to $\newG$. For example, the stationary distribution of $\newG$ is $\newpi$, its hitting times are $\newH(u,v)$, and so on.

\begin{lemma}
\label{lemma:caterpillarify-hitting-times-v2}
Let $G \in \tree{n}{d}$ be a non-caterpillar with leaf-to-leaf path $P = \{ v_0, v_1, \ldots, v_s \}$, where $s \leq d$, and leaf $y \in V(G \backslash P)$ that is not adjacent to $P$. If $\newG = \sigma(G,P;y)$ then the following statements hold.
\begin{enumerate}[label=(\alph*)]
\item $\newH(v_i, v_j)=H(v_i,v_j)$, for all $0 \leq i,j \leq s$.
\item More generally, for $0 \leq j \leq s$, we have $\newH(v,v_j) \leq H(v,v_j)$ for all $v \in V(G)$, 
with  $\newH(y,v_j) < H(y,v_j)$, in particular.
\item $\newH(\newpi, v_j) < H(\pi, v_j)$ for $0 \leq i \leq s$. 
\end{enumerate}
\end{lemma}

\begin{proof}
Let $x \notin V(P)$ be the unique neighbor of $y$ in $G$, and let  $v_k$ be the closest $P$-vertex to $y$ (and to $x$). We have $\newG  = \sigma(G,P; y) = G - (x,y) + (v_k,y)$. 

(a) By equation \eqref{eqn:hitting-time}, we have  $\newH(v_i, v_j)=H(v_i,v_j)$ for all $0 \leq i,j \leq s$.  

(b) Moving $y$ closer to $v_j$ cannot increase hitting times to $v_j$, so $\newH(v,v_j) \leq H(v,v_j)$ for all $v \in V(G)$. In particular, we have 
$$\newH(y,v_j) = 1 + \newH(v_k,v_j) < 1 + H(x,v_k) + H(v_k, v_j)  = H(y,v_j).$$

(c) Now let's show that $\newH(\newpi, v_j) < H(\pi, v_j)$.  When we change $G$ into $\newG$, we have $\newdeg(v_k) = \deg(v_k)+1$ and $\newdeg(x) = \deg(x)-1$, while all other degrees remain the same.
As noted above, the  hitting times to $v_j$ are non-increasing and the only $\pi_w$ that change are $\pi_{v_k}$ and $\pi_x$. Therefore,
\begin{align*}
& \quad \newH(\newpi, v_j) - H(\pi,v_j)   \\
&= \sum_v \newpi_v \newH(v,v_j)  - \sum_v \pi_v H(v,v_j)    \\
&\leq \newpi_{v_k} \newH(v_k,v_j) + \newpi_x \newH(x,v_j) - \pi_{v_k} H(v_k,v_j) + \pi_x H(x,v_j) \\
&= \frac{1}{2(n-1)} \bigg(
\newdeg(v_k) \newH(v_k,v_j) + \newdeg(x) \newH(x,v_j) \\
& \qquad \qquad - \deg(v_k) H(v_k,v_j) - \deg(x) H(x,v_j)  \bigg)\\
&\leq  \frac{1}{2(n-1)} \left(  H(v_k,v_j) - H(x,v_j) \right) < 0,
\end{align*}
where the first inequality follows from statement (b). 
\end{proof}

\begin{lemma}
\label{lemma:caterpillarify-mixing-time-step1}
Let $G \in \tree{n}{d}$ be a non-caterpillar with pessimal path $P = \{ v_0, v_1, \ldots, v_s \}$, where $s \leq d$, and leaf $y \in V(G \backslash P)$ that is not adjacent to $P$. If $\newG = \sigma(G,P;y)$ then
$\Tmix(G) = H(v_0, \pi) < \newH(v_0, \newpi) \leq \Tmix(\newG).$
\end{lemma}

Note that we have $s<d$ when the pessimal path $P$ is not a geodesic. 

\begin{proof}
Using Lemma \ref{lemma:caterpillarify-hitting-times-v2} (c) for $j=s$, we have
\begin{align*}
\newH(v_0, \newpi) - H(v_0, \pi) &=
\newH(v_0,v_s) - \newH(\newpi,v_s) - (H(v_0,v_s) - H(\pi, v_s))  \\
&=
(\newH(v_0,v_s) - H(v_0,v_s) ) + (H(\pi, v_s)  - \newH(\newpi,v_s)) \\
&=
-(\newH(\newpi,v_s) -H(\pi, v_s)  )> 0.
\end{align*}
Finally, we have
$$
\Tmix(\newG) \geq \newH(v_0, \pi^*) > H(v_0, \pi) = \Tmix(G).
$$
\end{proof}

\begin{lemma}
\label{lemma:caterpillarify-mixing-time}
Let $G \in \tree{n}{d}$ be a tree with pessimal path $P = \{ v_0, v_1, \ldots, v_s \}$, and leaf $y \in V(G \backslash P)$ that is not adjacent to $P$. Then there exists a caterpillar  $\newG \in \tree{n}{s}$, where $s \leq d$, such that
$$
\Tmix(\newG) \geq \newH(v_0, \newpi) > H(v_0,\pi) = \Tmix(G).
$$
\end{lemma}

\begin{proof}
Let $P=\{ v_0, v_1, \ldots, v_s \}$ be a pessimal path of $G$. We transplant each leaf that is not adjacent to $P$, one at a time, so that each leaf is now adjacent to $P$. After each leaf transplant,  $P$ is still the pessimal path due to 
Lemma \ref{lemma:caterpillarify-hitting-times-v2} (a) and (b). Meanwhile, the mixing time strictly increases by Lemma \ref{lemma:caterpillarify-mixing-time-step1}. We repeat this process until every leaf is adjacent to $P$. We now have a caterpillar $\newG \in \tree{n}{s}$, where $s \leq d$, such that $\Tmix(G) < \Tmix(\newG)$.
\end{proof}

\subsection{Phase Two: Caterpillar to Double Broom or Near Double Broom}

In this subsection, we show that if $G$ is a caterpillar, then can move leaves in pairs until we either have a double broom or a near double broom (see Definition~\ref{def:near-double-broom}). With each change, the mixing time increases. We will need the following notation for partitioning a caterpillar along its spine, as shown in  Figure \ref{fig:Vk-example}.

\begin{definition}
Consider a caterpillar $G$ with spine $P=\{v_0, v_1, \ldots, v_d\}$. For $0 \leq k \leq d$, define the set $V_k$ to be the vertex  $v_k$ and the leaves in $G \backslash P$ that are adjacent to it. 
\end{definition}

\begin{figure}[ht]

\begin{center}
\begin{tikzpicture}[scale=1.1]

\begin{scope}

\foreach \x in {0,2,5}
{
\draw[fill] (\x,0) circle (2pt);
\node[above] at (\x,0) {$v_{\x}$};
\draw (\x,0 ) ellipse (.25 and .55);
\node at (\x, 0.85) {$V_{\x}$};
}

\foreach \x in {1,3,4}
{
\draw[fill] (\x,0) circle (2pt);
\node[above] at (\x,0) {$v_{\x}$};
\node at (\x, 0.85) {$V_{\x}$};
}

\draw (1,-.5) ellipse (.4 and 1.05);
\draw (3,-.5) ellipse (.5 and 1.05);
\draw (4,-.5) ellipse (.3 and 1.05);

\draw[thick] (0,0) -- (5,0);

\begin{scope}[shift={(3,0)}]

\foreach \x in {-74,-90,-104}
{
\draw[thick] (0,0) -- (\x:1);
\draw[fill] (\x:1) circle (2pt);
}

\end{scope}

\begin{scope}[shift={(1,0)}]

\foreach \x in {-82,-98}
{
\draw[thick] (0,0) -- (\x:1);
\draw[thick, fill] (\x:1) circle (2pt);
}

\end{scope}

\draw[thick] (4,0) -- (4,-1);
\draw[thick, fill] (4,-1) circle (2pt);

\end{scope}

\end{tikzpicture}
\end{center}

\caption{Partitioning a caterpillar in $\tree{12}{5}$ into sets $V_0, V_1, V_2, V_3, V_4, V_5$.}
\label{fig:Vk-example}
\end{figure}
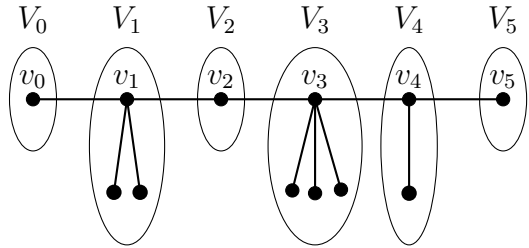

We are now ready to describe the effect of performing the tree surgeries shown in Figure \ref{fig:cat-to-double-broom}.

\begin{lemma}
\label{lemma:two-leaf-hitting-time}
Given a caterpillar $G$ with spine $P=\{v_0, v_1, \ldots, v_d\}$ that has leaves adjacent to $v_i$ and $v_j$ where $2\leq i \leq j \leq d-2$. Let $\newG = \tau(G,P; (i,i-1) \wedge (j,j+1))$, where $2 \leq i \leq j \leq d-2$. 

For $v \in V_k$, set $\triangle H(v,v_d) = \newH(v,v_d) - H(v,v_d)$. Then
$$
\triangle H(v, v_d) = \left\{
\begin{array}{cl}
0 & 0 \leq k \leq i-1 \mbox{ or } j+1 \leq k \leq d, \\
2+H(i-1,i) & v = x, \\
-2  & i \leq k \leq j \mbox{ and } v \notin \{x,y\}, \\
2-H(j,j+1) & v = y. \\
\end{array}
\right.
$$
\end{lemma}

\begin{proof}
All four cases follow from equation \eqref{eqn:hitting-time} and the fact that 
\begin{equation}
\label{eqn:path-sum}
H(v_s, v_t) = H(v_s, v_{s+1})+ H(v_{s+1}, v_{s+2}) +\cdots + H(v_{t-1}, v_t) \quad \mbox{for } 0 \leq s < t \leq d.
\end{equation} 
We leave the elementary calculations to the reader.
\end{proof}

\begin{lemma}
\label{lemma:two-leaf-pi-access-time}
Given a caterpillar $G$ with spine $P=\{v_0, v_1, \ldots, v_d\}$ and leaf $x$ adjacent to $v_i$ and leaf $y$ adjacent to $v_j$, where $2 \leq i \leq j \leq d-2$. Let $\newG = \tau(G,P; (i,i-1) \wedge (j,j+1))$.
Then $\newH(\newpi,v_d) < H(\pi,v_d)$ and $\newH(\newpi,v_0) < H(\pi,v_0)$.
\end{lemma}

\begin{proof}
We prove the statement for $v_d$; the proof for $v_0$ follows by symmetry.
We show that
$$
2|E| (\newH(\newpi, v_d) - H(\pi,v_d) )= \sum_{v \in V} \newdeg(v) \newH(v,v_d) - \sum_{v 
\in V} \deg(v) H(v,v_d) < 0.
$$
We calculate $ \newdeg(v) \newH(v,v_d) - \deg(v) H(v,v_d)$ for eight types of vertices. We rely on  Lemma \ref{lemma:two-leaf-hitting-time} and  equations  \eqref{eqn:hitting-time} and  \eqref{eqn:path-sum} to justify the calculations below.

\begin{enumerate}
\item For $v \in V_k$ where $0 \leq k \leq i-1 \mbox{ or } j+1 \leq k \leq d$ and $v \notin \{ v_{i-1}, v_{j+1} \}$,
$$
 \newdeg(v) \newH(v,v_d) -  \deg(v) H(v,v_d) =  \deg(v) H(v,v_d) -  \deg(v) H(v,v_d) = 0.
$$

\item For $v = v_{i-1}$, we have
\begin{align*}
&\quad \newdeg(v_{i-1})\newH(v_{i-1},v_d) -  \deg(v_{i-1})H(v_{i-1},v_d) \\
 &=
(\deg(v_{i-1})+1) H(v_{i-1},v_d) -  \deg(v_{i-1}) H(v_{i-1},v_d) \\
 &=
 H(v_{i-1},v_d).   
\end{align*}

\item For $v=x$, we have
\begin{align*}
 \newdeg(x) \newH(x,v_d) -  \deg(x) H(x,v_d) 
 &=
 1 + \newH(v_{i-1},v_d) - (1 + H(v_i,d))  \\
 &=
\newH(v_{i-1},v_i) + \newH(v_i,v_d) -  H(v_i,v_d) \\
&=
 2 + H(v_{i-1},v_i). 
\end{align*}

\item For $v=v_i$, we have
\begin{align*}
& \quad \newdeg(v_{i}) \newH(v_{i},v_d) -  \deg(v_{i}) H(v_{i},v_d) \\
 &=
(\deg(v_{i})-1)H(v_{i},v_d) -  \deg(v_{i})H(v_{i},v_d) \\
 &=
-  H(v_{i},v_d). 
\end{align*}

\item For $v \in V_k$ where $i \leq k \leq j$ and $v \notin \{v_i, x, v_j, y \}$, we have
\begin{align*}
 \newdeg(v) \newH(v,v_d) -  \deg(v) H(v,v_d) 
 &=  \deg(v) ( H(v,v_d) -2 - H(v,v_d))  \\
 &= -2 \deg(v).
\end{align*}
When we sum all of these values, we obtain
$$
- 2\sum_{k=i}^j \sum_{v \in V_k  \backslash \{ x, v_i, v_j, y\}}  \!\!\!\!\!\!\!\!\!\!\! \deg(v)
=
 - 2 \big(H(v_j, v_{j+1}) - H(v_{i-1}, v_i) - \deg(v_j)  + \deg(v_i) \big).
$$

\item For $v = y$, we have
\begin{align*}
 & \quad \newdeg(y) \newH(y,v_d) -  \deg(y) H(y,v_d) \\
 &=
  1 + \newH(v_{j+1},v_d) - (1 + H(v_j,d)) \\
 &=
\newH(v_{j+1},v_d) - H(v_j,v_{j+1}) -  H(v_{j+1},v_d)  =
-   H(v_{j},v_{j+1}). 
\end{align*}

\item For $v=v_j$, we have
\begin{align*}
& \quad \newdeg(v_{j}) \newH(v_{j},v_d) -  \deg(v_{j}) H(v_{j},v_d)  \\
 &=
 (\deg(v_{j})-1)(\newH(v_{j},v_{j+1}) + \newH(v_{j+1}, v_d) )-  \deg(v_{j}) H(v_{j},v_d) 
 \\
 &=
(\deg(v_{j})-1)(H(v_{j},v_{j+1}) -2 + H(v_{j+1}, v_d) )-  \deg(v_{j}) H(v_{j},v_d)  \\
 &=
 2 - 2\deg(v_{j}) -H(v_j,v_d).
\end{align*}

\item For $v=v_{j+1}$, we have
\begin{align*}
& \quad \newdeg(v_{j+1}) \newH(v_{j+1},v_d) -  \deg(v_{j+1}) H(v_{j+1},v_d)  \\
 &=
(\deg(v_{j+1})+1)\newH(v_{j+1}, v_d) -  \deg(v_{j+1}) 
 H(v_{j+1},v_d) 
 \\
 &=
 H(v_{j+1},v_d).
\end{align*}

\end{enumerate}

The overall change $\newH(\newpi, v_d) - H(\pi,v_d) $ is $1/2|E|$ times the sum of these eight terms. This sum is 
\begin{align*}
&\quad 
H(v_{i-1},v_d) 
+ 2 + H(v_{i-1},v_i) 
- H(v_{i},v_d)  \\
& \qquad
-2 H(v_j, v_{j+1}) +2 H(v_{i-1}, v_i) +2 \deg(v_j)  -2 \deg(v_i) \\
& \qquad
-   H(v_{j},v_{j+1}) 
+2 -2 \deg(v_j) - H(v_{j},v_d)
+H(v_{j+1},v_d)
\\
&= 
H(v_{i-1},v_d) 
- H(v_{i},v_d) 
+3 H(v_{i-1}, v_i) 
 - H(v_{j},v_d)
+H(v_{j+1},v_d)  \\
& \qquad
 -3 H(v_j, v_{j+1}) +4 -2 \deg(v_i)  
\\
&= 
4 H(v_{i-1}, v_i) 
-4 H(v_j, v_{j+1})  + 4 -2 \deg(v_i)  < 0.
\end{align*}
This confirms that $\newH(\newpi, v_d) - H(\pi,v_d)$.
\end{proof}

\begin{corollary}
\label{cor:two-leaf-mixing-time}
Given a caterpillar $G$ with spine $P=\{v_0, v_1, \ldots, v_d\}$ with leaf $x$ adjacent to $v_i$ and leaf $y$ adjacent to $v_j$, where $2 \leq i \leq j \leq n-2$. Let $\newG = \tau(G,P; (i,i-1) \wedge (j,j+1))$.
Then $\Tmix(\newG) > \Tmix(G)$.
\end{corollary}

\begin{proof}
By Lemma \ref{lemma:two-leaf-hitting-time}, we have $\newH(v_0,v_d) = H(v_0,v_d)$ and by Lemma \ref{lemma:two-leaf-pi-access-time}, we have
$\newH(\newpi,v_d) < H(\pi,v_d)$. Therefore
$$
\Tmix(\newG) \geq \newH(v_0, v_d) - \newH(\newpi, v_d)
> H(v_0, v_d) - H(\pi, v_d) = \Tmix(G).
$$
\end{proof}

\begin{lemma}
\label{lemma:to-appendix-broom}
Given a caterpillar $G \in \tree{n}{d}$, with at least two leaves adjacent to spine vertices $\{ v_2, v_3, \ldots, v_{d-2}\}$. Then there is another caterpillar $\newG \in \tree{n}{d}$ such that $\Tmix(G) < \Tmix(\newG$), where $\newG$ is either a double broom or  a near double broom. 
\end{lemma}

\begin{proof}
By Corollary \ref{cor:two-leaf-mixing-time}, we can perform tree surgery $\tau(G,P; (i,i-1) \wedge (j,j+1))$ where $2 \leq i \leq j \leq d-2$ to obtain $\newG$ with $\Tmix(\newG) > \Tmix(G)$. We can repeat this process as long as there are at least two leaves adjacent to $\{ v_2, v_3, \ldots, v_{d-2}\}$, and the mixing time increases with each surgery. When we are done we either have a double broom, or a near double broom with one additional leaf adjacent to $\{ v_2, v_3, \ldots, v_{d-2}\}$.
\end{proof}

\subsection{Phase Three: Creating a Balanced Double Broom}

In this section, we show that when we have a near double broom or an  unbalanced double broom, then we can move leaves to the smaller end of the double broom, and increase the mixing time. Without loss of generality, our graph has spine $\{v_0, v_1, \ldots, v_d \}$, and there are at least as many  leaves adjacent to $v_1$ as there are to $v_{d-1}$. 

\begin{lemma}
\label{lemma:near-double-broom}
Let $G \in \tree{n}{d}$ be a near double broom with spine $v_0, v_1, \ldots, v_d$, where $\deg(v_1) \geq \deg(v_{n-1})$, and with one more leaf $x$ incident with $v_i$ where $2 \leq i \leq  d-2$. Let $\newG$ be the double broom obtained via the leaf transplant $\tau(G,P; (i,d-1))$. Then 
$\Tmix(\newG) > \Tmix(G)$. 
\end{lemma}

\begin{proof}
First, we observe that
\begin{equation}
\label{eqn:near-double-broom1}
\newH(v_0, v_d) - H(v_0,v_d) = -2(d-i-1).
\end{equation}
by equation \eqref{eqn:hitting-time}, since we have moved leaf $x$ a distance of $d-i-1$ closer to $v_d$. By equation \eqref{eqn:mixing-time-formula}, we will have $\Tmix(\newG) > \Tmix(G)$ if and only if 
\begin{equation}
\label{eqn:near-double-broom2}
 \newH(v_0, v_d) - H(v_0,v_d) > \newH(\newpi, v_d) - H(\pi,v_d) = 
\sum_{v \in V}  \big( \newpi_v \newH(v,v_d) - \pi_v H(v,v_d) \big).
\end{equation} 
 It will be more convenient to multiply by $2|E|$ so that the summands on the right hand side become
 $\newdeg(v) \newH(v,v_d) - \deg(v) H(v,v_d)$. There are six different cases for $v \in V$.

\begin{enumerate}
\item For $v \in V_k$ for $0 \leq k \leq i$ and $v \notin \{ v_i, x\}$, we have
\begin{align*}
&\quad
 \newdeg(v) \newH(v,v_d) - \deg(v) H(v,v_d) \\
&= 
\deg(v) (H(v,v_d) - 2(d-i-1)) -\deg(v) H(v,v_d) \\
&= -2 \deg(v)(d-i-1).
\end{align*}
Summing over $V_k$ where $0 \leq k \leq i$ and $v \notin \{ v_i, x\}$ gives
\begin{align*}
&\quad -2(d-i-1) (H(v_i, v_{i+1}) - \deg(v_i) - \deg(x)) \\
&= -2(d-i-1) (H(v_i, v_{i+1}) - \deg(v_i) - 1).
\end{align*}

\item For $v=v_i$, we have
\begin{align*}
&\quad
 \newdeg(v_i) \newH(v_i,v_d) - \deg(v) H(v_i,v_d) \\
 &=
 (\deg(v_i)-1) (H(v_i,v_d) -2(d-i-1)) -\deg(v_i) H(v_i,v_d) \\
 &= - H(v_i,v_d) -2(d-i-1) (\deg(v_i)-1).
\end{align*}

\item For $v=x$, we have
\begin{align*}
 \newdeg(x) \newH(x,v_d) - \deg(x) H(x,v_d) 
 &=
 (1+\newH(v_{d-1},v_d)) - (1+H(v_i,v_d)) \\
 &= H(v_{d-1},v_d))- H(v_i,v_d) \\
 &= -H(v_i, v_{d-1}).
\end{align*}

\item For $v \in V_k$ where $i+1 \leq k \leq d-2$, we have
\begin{align*}
 & \quad \newdeg(v) \newH(v,v_d) - \deg(v) H(v,v_d)  \\
 &=
\deg(v) (H(v,v_d)- 2( d-k-1)) - \deg(v) H(v,v_d) \\
&=- 2( d-k-1) \deg(v) =- 4( d-k-1)  .
\end{align*}
Summing over  $i+1 \leq k \leq d-2$ yields
\begin{align*}
&\quad -4 \sum_{k=i+1}^{d-2}  (d-1) + 4 \sum_{k=i+1}^{d-2} k =-2  (d-i-2) (d-i-1).
\end{align*}

\item For $v = v_{d-1}$, we have
\begin{align*}
&\quad
\newdeg(v_{d-1}) \newH(v_{d-1},v_d) - \deg(v_{d-1}) H(v_{d-1},v_d)  \\
&=
(\deg(v_{d-1})+1) H(v_{d-1},v_d)   - \deg(v_{d-1}) H(v_{d-1},v_d)   \\
&=
H(v_{d-1},v_d).
\end{align*}

\item For $v \in V_{d-1} \backslash \{ v_{d-1} \}$, we have
$$
\newdeg(v) \newH(v,v_d) - \deg(v) H(v,v_d) =0.
$$

\end{enumerate}

The overall change $\newH(\newpi, v_d) - H(\pi,v_d) $ is $1/2|E|$ times the sum of these six terms. This sum is 
\begin{align*}
&\quad
-2(d-i-1) (H(v_i, v_{i+1}) - \deg(v_i) - 1)
- H(v_i,v_d) 
\\
&\qquad
 -2(d-i-1) (\deg(v_i)-1)
-H(v_i, v_{d-1})
\\
&\qquad
-2  (d-i-2) (d-i-1)
+ H(v_{d-1}, v_d)
\\
&=
-2(d-i-1) (H(v_i, v_{i+1})  - 2)
- 2H(v_i,v_{d-1}) 
-2  (d-i-2) (d-i-1) 
\\
&=
-2(d-i-1) \bigg(H(v_i, v_{i+1})  +d-i- 4 \bigg)
- 2H(v_i,v_{d-1}) .
\end{align*}

Multiplying inequality \eqref{eqn:near-double-broom2} by $-2|E|$ and using equation \eqref{eqn:near-double-broom1}, we must show that
\begin{equation}
\label{eqn:throw-leaf}
4|E|(d-i-1) \leq 2(d-i-1) \bigg(H(v_i, v_{i+1})  +d-i- 4 \bigg) + 2H(v_i,v_{d-1}).
\end{equation}
Observe that we can write $H(v_i,v_{d-1})$ in terms of $H(v_i,v_{i+1})$ as follows:
\begin{align*}
H(v_i,v_{d-1}) &= \sum_{k=i}^{d-2} H(v_k,v_{k+1}) 
=
\sum_{k=0}^{d-i-2}  \left( H(v_i,v_{i+1}) + 2k \right) \\
&= (d-i-1) H(v_i,v_{i+1}) + 2\sum_{k=0}^{d-i-2} k \\
&= (d-i-1) H(v_i,v_{i+1}) + (d-i-2)(d-i-1).
\end{align*}
So we can divide our desired inquality \eqref{eqn:throw-leaf} by $2(d-i-1)$ to obtain the equivalent condition
\begin{align}
\nonumber
2|E| &\leq H(v_i,v_{i+1}) + d-i-4 + H(v_i,v_{i+1}) + d-i-2 \\
\label{eqn:throw-leaf2}
 |E| &\leq H(v_i,v_{i+1}) +d-i-3.
\end{align}
Now let's explain why  inequality \eqref{eqn:throw-leaf2} holds. 
First, we define
$$
L = \bigcup_{k=0}^i V_k = V_{v_i: v_{i+1}}
$$
to be the vertices to the left of spine vertex $v_{i+1}$. 
By equation 
\eqref{eqn:adjhtime}, we have
$$
H(v_i,v_{i+1})  = \sum_{v \in L} \deg(v) = 2|L| -1.
$$
If $|L| \geq (n+1)/2$, then inequality \eqref{eqn:throw-leaf2} holds. 
Otherwise $|L| < (n+1)/2$, so we must have  $i \leq d/2$ because the left end of the tree has $\ell+1 > r$ leaves. So we can define
$$ M = \bigcup_{k=i+1}^{d-i-1} V_k \qquad \mbox{and} \qquad R = \bigcup_{k=d-i}^{d} V_k,
$$
corresponding to the  middle and right end of the tree. 
We have $|M| = d-2i+1$ and  $|R| < |L|$, and therefore
\begin{align}
\label{eqn:throw-last-step}
|E| &= |L| + |M| + |R| -1 < 2|L| -1 + (d-2i+1) \leq H(v_i, v_{i+1}) + d -i -3
\end{align}
because $i \geq 2$.
Therefore inequality \eqref{eqn:throw-leaf2} holds, and the proof is complete.
\end{proof}

\begin{corollary}
\label{cor:more-balanced-double-broom}
Let $G \in \double{n}{d}$ be a double broom with $\ell$ left leaves and $r$ right leaves, where $\ell \geq r+2$. 
Then the double broom $\newG$ obtained by leaf transplant $\tau(G,P; (1,d-1))$ satisfies $\Tmix(\newG) > \Tmix(G)$.
\end{corollary}

\begin{proof}
Treat the  double broom $G$ as having $\ell+1$  vertices adjacent to $v_1$ and $r$  vertices adjacent to $v_{d-1}$. Prior to equation \eqref{eqn:throw-last-step}, the argument is identical to the proof of  Lemma \ref{lemma:near-double-broom}, with $i=1$. Our desired condition \eqref{eqn:throw-leaf2} becomes
$|E| \leq H(v_1,v_2) + d-4$. We know that $|R| \leq |L| -2$ and we have $|M|=d-2$. Therefore
$$
|E| = |L| + |M| + |R| -1 \leq 2|L|-2 +d-2 -1 = 2|L|-1 + d-4
= H(v_1,v_2) + d-4.
$$
The resulting double broom $\newG$ has $\ell$ left leaves and $r+1$ right leaves, and its mixing time is larger.
\end{proof}

\subsection{Proof of Theorem \ref{thm:broom-max-mixing}}
\label{sec:end}

We now present the proof of our main theorem.

\begin{lemma}
\label{lemma:broom-max-mixing}
For all $d \geq 3$, the quantity $\max_{G \in \Tnd} \Tmix(G)$ is achieved by the balanced double broom $\Dnd$.
\end{lemma}

\begin{proof}
We give a strong induction proof on the diameter $d$ of the tree. Starting with the base case $d=3$, every tree  $G \in \tree{n}{3}$ is a caterpillar. In fact, each $G$ is also a double broom. 
If $G \neq B_{n,3}$, is not a balanced double broom, then  by Corollary \ref{cor:more-balanced-double-broom}, making it closer to balanced increases the mixing time. We can repeat this process until we have a balanced double broom, and therefore 
$$\Tmix(B_{n,3}) = \max_{G \in \tree{n}{3}} \Tmix(G).$$

Assume the statement holds for trees (of all orders) with diameter at most $d-1$. 
Now consider a  tree $G \in \tree{n}{d}$ that is not the double broom $\Dnd$. If $G$ is not a caterpillar, then our first step is to replace $G$ with a caterpillar with a larger mixing time. Let $P= \{ v_0, v_1, \ldots v_s \}$ be a pessimal path of $G$ where $s \leq d$.  By Lemma \ref{lemma:caterpillarify-mixing-time}, we can create a caterpillar $\newG \in \tree{n}{s}$ with $\Tmix(\newG) > \Tmix(G)$. 

So now assume that $G \in \tree{n}{s}$ is a caterpillar where $s \leq d$. If there are at least two leaves adjacent to $\{v_2, v_3, \ldots, v_{s-2}\}$, then we use Lemma \ref{lemma:to-appendix-broom} (with $d=s$) to create tree $\newnewG \in \tree{n}{s}$ where $\Tmix(\newnewG) > \Tmix(G)$. The tree $\newnewG$ is either a double broom, or it is a near double broom with one  leaf adjacent to $\{v_2, \ldots, v_{s-2}\}$. In the latter case, we use Lemma \ref{cor:more-balanced-double-broom} to create a double broom with a larger mixing time. 

We are now in the endgame of our process. If our double broom is not balanced, we repeatedly apply Corollary \ref{cor:more-balanced-double-broom} to create the balanced broom $B_{n,r}$, increasing the mixing time at each step. Finally, if $s < d$, then Lemma \ref{cor:balanced-double-brooms-increasing} shows that $\Tmix(B_{n,r}) < \Tmix(B_{n,d})$.

The mixing time is strictly increasing during this process, so we know that $\Tmix(G) < \Tmix(B_{n,d})$ for any other tree $G \in \tree{n}{d}$. Therefore $\Tmix(\Dnd) = \max_{G \in \tree{n}{d}} \Tmix(G).$ 
\end{proof}

Finally, we prove our main theorem.

\begin{proof}[Proof of Theorem \ref{thm:broom-max-mixing}]
 Lemma \ref{lemma:broom-max-mixing} shows that the balanced double broom $\Dnd$ is the unique tree that achieves
 $\max_{G \in \tree{n}{d}} \Tmix(G)$. The mixing time formulas \eqref{eqn:broom-max-mixing-odd} and \eqref{eqn:broom-max-mixing-even} for $\Tmix(\Dnd)$ were proven in Corollary \ref{cor:balanced-double-broom-mixing-time}.  
\end{proof}

\section{Conclusion}
\label{sec:conclusion}

We have shown that the balanced double broom $\Dnd$ achieves the maximum mixing time among the trees in $\tree{n}{d}$ with $n$ vertices and diameter $d$.
The formula of Theorem \ref{thm:mixing-time} reveals that the mixing time of the balanced double broom is $\Theta(n d)$. More precisely, when the diameter $d=d(n)$ is a function of $n$, we have
$$
\Tmix (D_{n, d(n)}) = 
\begin{cases}
\frac{d-2}{2} n + o(n) & \mbox{when $d$ is constant}, \\
\frac{1}{2} n \, d(n)+ o(n \, d(n)) & \mbox{when $d(n) = o(n)$ is sublinear}, \\
\frac{3c-c^3}{6}n^2 + o(n^2) & \mbox{when } d(n) = c\,n + o(n) \mbox{ where } 0 < c \leq 1.
\end{cases}
$$

There are many open questions about random walks on the family $\tree{n}{d}$ of trees of order $n$ and diameter $d$. 
For example, it would be interesting to characterize the behavior of  the \emph{reset time} $\Treset = \sum_{v \in V} \pi_v H(v, \pi)$ and thee \emph{best mixing time} $\Tbestmix = \min_{v \in V} H(v,\pi)$.
Another direction would be to identify  the \emph{minimizing} structure in $\tree{n}{d}$ for each of these exact mixing measures. 
Finally, we could ask these same questions about the family of all graphs of diameter $d$: which graphs achieve the extremal values of each of these mixing measures?

\bibliographystyle{plain} 
\bibliography{biblio}

\end{document}